\documentclass[12pt]{amsart}

\usepackage{amsmath,amsthm,amsfonts,amssymb}

\newtheorem{thm}{Theorem}[section]
\newtheorem{lemma}[thm]{Lemma}
\newtheorem{lem}[thm]{Lemma}
\newtheorem{prop}[thm]{Proposition}

\newtheorem{cor}[thm]{Corollary}
\newtheorem{remark}[thm]{Remark}
\newtheorem{question}[thm]{Question}
\newtheorem{example}[thm]{Example}

\author[Mitja Mastnak]{Mitja Mastnak$^{\star}$}
\thanks{$^{\star}$ Supported by the Natural Sciences and Engineering Research Council (NSERC) of Canada, Discovery Grants 371994-2014 and 8809-2002-RGPIN}
\address{Department of Mathematics and Computing Science,
Saint Mary's University, 923 Robie Street, Halifax, NS, B3H3C3, Canada}
\email{mmastnak@cs.smu.ca}

\author{Heydar Radjavi$^{\star}$}
\address{
Department of Pure Mathematics,
University of Waterloo,
Waterloo, Ontario,
Canada N2L3G1, Canada}
\email{hradjavi@uwaterloo.ca}

\title[Semigroups whose ring commutators have real spectra]{Matrix semigroups whose ring commutators have real spectra are realizable}

\def\Mn{\mathcal{M}_n(\mathbb{C})}
\def\M{\mathcal{M}}
\def\S{\mathcal{S}}
\def\G{\mathcal{G}}
\def\J{\mathcal{J}}
\def\C{\mathcal{C}}
\def\H{\mathcal{H}}
\def\K{\mathcal{K}}
\def\CC{\mathbb{C}}
\def\RR{\mathbb{R}}
\def\A{\mathcal{A}}

\def\H{\mathcal{H}}
\def\D{\mathcal{D}}
\def\Cx{\mathbb{C}^{\times}}
\def\PP{\mathcal{P}}
\def\A{\mathcal{A}}
\def\avg{\operatorname{avg}}
\def\Pat{\operatorname{Pat}}
\def\tr{\operatorname{tr}}
\def\V{\mathcal{V}}
\def\N{\mathcal{N}}

\def\diag #1{\operatorname{diag}\left(#1\right)}

\begin{document}
\begin{abstract}
We study matrix semigroups in which ring commutators have real spectra.
We prove that irreducible semigroups with this property are
simultaneously
similar to semigroups of real-entried matrices.  We also obtain a
structure theorem for
compact groups satisfying the property under investigation.
\end{abstract}
\maketitle
\section{Introduction}
Let $\G$ be an irreducible group of complex matrices, that is, when viewed as linear operators on a finite-dimensional complex vector space, the members of $\G$ have no common invariant subspace other then $\{0\}$ and the whole space.
There are certain known conditions under which $\G$ is realizable, i.e., $\G$ is simultaneously similar to a group of real matrices.  For example, let $\varphi$ be a rank-one functional on the algebra $\mathcal{M}_n(\mathbb{C})$ of all $n\times n$ complex matrices (in other words, $\varphi(M)=\tr(TM)$ for all $M$ in $\M_n(\mathbb{C})$, where $T$ is a fixed matrix of rank one).  If $\varphi(\G)\subseteq\mathbb{R}$, then $\G$ is realizable (see \cite{BMR}, \cite{BMR1}, and \cite{RY}).

It follows from \cite{B} that if the spectra of members of $\G$ are all real, then $\G$ is realizable.  We consider the effect of weaker hypotheses: What can we get, for example, if we merely assume that the members of the commutator subgroup have real spectra?  For compact groups, this is equivalent to the assumption that the commutator subgroup consists of involutions.  In this case we conclude that $\G$ is essentially a signed permutation group with commutative pattern.  There is a weaker hypothesis whose effect we have not been able to ascertain: What if we know only that every commutator is an involution?

It is interesting that if ring commutators are considered, as opposed to group commutators, then the corresponding weak assumption on $\G$ gives the desired result: If $AB-BA$ has real spectrum for every $A$ and $B$ in a compact group $\G$, then $\G$ is realizable and finite.  Furthermore $\G$ has a very simple structure given in Theorem \ref{thm-maingrp}.  We also consider  some semigroups whose ring commutators have real spectra.
\section{Preliminaries}
\subsection{Monomial matrix groups}
A subspace $U$ of $\mathbb{C}^n$ is called a {\em standard subspace} if it is spanned by a subset of the standard basis $(e_i)_{i=1}^n$, where $e_i=(0,\ldots, 0,1,0,\ldots ,0)$ with $1$ at the $i$-th position.  A matrix is called an {\em indecomposable matrix} if it has no nontrivial standard invariant subspaces.  A set of matrices is called indecomposable if it has no nontrivial common standard invariant subspaces.  These notions are usually discussed in the context of non-negative-entried matrices, but in this note the notions will also be studied for more general matrices and sets of matrices.

We say that an invertible matrix is {\em monomial} or a {\em weighted permutation} if it has exactly one nonzero entry in each row (or, equivalently, exactly one nonzero entry in each column).  The nonzero entries are often referred to as {\em weights}.  If all weights are equal to $1$ then a matrix is referred to as a {\em permutation}, and if all the weights belong to $\{\pm 1\}$, then we call the matrix in question a {\em signed permutation}.

We say that a set of matrices is monomial if every member is monomial.  We say that a group of matrices is a (signed) permutation group if every member is a (signed) permutation. We say that a set of matrices is {\em monomializable} if it is simultaneously similar to a set of monomial matrices.  

The {\em pattern} $\Pat(A)$ of a monomial matrix $A$ is the permutation matrix obtained by replacing all nonzero entries in $A$ by $1$'s.  The pattern $\Pat(\G)$ of a monomial matrix group $\G$ is the permutation matrix group obtained by replacing every member of $\G$ by its pattern.  We say that $\G$ has commutative pattern if its pattern group is commutative.  We remark that a monomial group $\G$ is indecomposable if and only if its pattern group acts transitively on the set $\{e_1,\ldots, e_n\}$.

We will frequently deal with {\em tensor products} of matrices.  Throughout the paper we use the canonical isomorphism $\mathcal{M}_{n_1}(\mathbb{C})\otimes \mathcal{M}_{n_2}(\mathbb{C})\stackrel{\sim}{\to}\mathcal{M}_{n_1 n_2}(\mathbb{C})$ given by identifying $A\otimes B$ with the 
$n_1\times n_1$ block matrix whose $(i,j)$-block is the $n_2\times n_2$ matrix $A_{ij}B$.  Tensor products of length $3$ or more are read from left to right, that is $A_1\otimes A_2\otimes\ldots \otimes A_k=A_1\otimes(A_2\otimes( \ldots \otimes (A_{k-1}\otimes A_k)\ldots )).$ 

For $n\in\mathbb{N}$, we use $C_n$ to denote the cycle matrix
$$
C_n = \begin{pmatrix} 0 & 0 & \ldots & 0 & 1 \\ 1 & 0 & \ldots & 0 & 0 \\
\vdots & \vdots &  & \vdots &\vdots\\
0 & 0 & \ldots & 1 & 0 \end{pmatrix}\in\M_n(\mathbb{C}),
$$
and we use $\C_n=\langle C_n \rangle\subseteq\Mn$ to denote the cyclic matrix group of order $n$ generated by $C_n$.

We will use $\D_n(\mathbb{C})$ to denote the set of all diagonal $n\times n$ complex matrices, $\D_n(\pm 1)$ to denote the group of all signed diagonal matrices, and $\D_n^+(\pm 1)$ to denote the set of all signed diagonal matrices of determinant $1$.

\begin{lemma}\label{transitive-commutative} Let  $\K\subseteq\Mn$ be a commutative monomial matrix group such that $I$ is the only diagonal element of $\K$.  Then $\K$ is indecomposable if and only if, up to monomial similarity, we have that
$$
\K=\C_{n_1}\otimes\ldots \otimes\C_{n_k}
$$
for some factorization $n=n_1\ldots n_k$.  
\end{lemma}
\begin{proof}
$\K$ is an abelian group acting transitively on the set of lines $\{\CC
e_1,\ldots, \CC e_n\}$.  The action is faithful since we have that $I$ is the only diagonal element of $\K$. 
A transitive faithful action of an abelian group cannot have nontrivial elements with fixed points and hence has to be isomorphic to the left regular action of the group on itself.

Let $$\K\simeq \K_1\times\ldots\times \K_k$$ be a decomposition of the (abstract) finite
abelian group $\K$ into cyclic subgroups $\K_i$, where $\K_i$ is a cyclic group generated by
$G_i\in\K$ of order $n_i$, $i=1,\ldots, k$. The action of $\G$ on $\{\mathbb{C} e_1,\ldots, \mathbb{C} e_n\}$ can be described as follows:  re-index the set $$\{e_i : i=1,\ldots, n\} $$ as $$\{e_{i_1,\ldots, i_k} : 1\le i_j\le n_j, j=1,\ldots, k\}.$$ The action of $G_1^{a_1}\ldots G_k^{a_k}\in \K$  on $\mathbb{C} e_{i_1,\ldots, i_k}$ gives $\mathbb{C} e_{i'_1,\ldots, i'_k}$, where $i'_j=i_j+a_j\;\mathrm{mod}\; n_j$ for $j=1,\ldots, k$.  If we identify 
$$
e_{i_1,\ldots,i_k} = e_{i_1}\otimes\ldots\otimes e_{i_k},
$$
then we have that for $j=1,\ldots, k$, the element $G_j\in\K$ is equal to
$$I_{n_1}\otimes\ldots\otimes I_{n_{j-1}}\otimes D_jC_{n_j}\otimes I_{n_{j+1}}\ldots\otimes I_{n_k}$$ for some diagonal matrix $D_j\in\M_{n_j}(\mathbb{C})$.  Note that $G_j^{n_j}=\det(D_j) I_n$, so that we must have $\det(D_j)=1$.
If for $j=1,\ldots, k$ we have that $D_j=\diag{d_1^{(j)},\ldots, d_{n_j}^{(j)}}$ with $d_1^{(j)}\cdots d_{n_j}^{(j)}=1$, then let $$X_j=\diag{1,d_1^{(j)},d_1^{(j)}d_2^{(j)},\ldots, d_1^{(j)}\cdots d_{n_j-1}^{(j)}},$$ and let $X=X_1\otimes X_2\otimes\ldots\otimes X_k$.   Now observe that for each $j=1,\ldots, k$ we have that $X^{-1}G_jX = I_{n_1}\otimes\ldots\otimes I_{n_{j-1}}\otimes C_{n_j}\otimes I_{n_{j+1}}\ldots\otimes I_{n_k}$.
\end{proof}

\subsection{Block monomial matrices and Clifford's Theorem}
We say that a group $\G\subseteq\M_n(\mathbb{C})$ of matrices is block monomial with respect to a decomposition $\mathbb{C}=\V_1\oplus\ldots\oplus\V_r$ if for every $G\in\G$ and every $i\in\{1,\ldots, r\}$ there is a $j\in\{1,\ldots, r\}$ such that $G\V_i\subseteq\V_j$.  For $i=1,\ldots, r$, let $P_i$ denote the projection to $\V_i$ with respect to the decomposition in question.  We call $G_{i,j}=P_j G P_i\subseteq \mathrm{L}(\V_i,\V_j)$ the $(i,j)$-block entry of $G$.  Note that $\G$ is block monomial if and only if in each block-row every element $G\in\G$ has exactly one nonzero block entry.  If $\V_i=\mathbb{C}e_i$, $i=1,\ldots, r$, then $\G$ is block monomial if and only if it is monomial.

The following result is well-known.  We include a sketch of the proof for completeness.
\begin{prop}\label{prop-block}
Let $\G\subseteq\M_n(\mathbb{C})$ be an irreducible group of matrices that is block-monomial with respect to some decomposition $\mathbb{C}^n=\V_1\oplus\ldots\oplus \V_r$, $r>1$, $\V_i\not=0$ for $i=1,\ldots, r$.  Let $P_1,\ldots P_r$, denote the projections to the corresponding summands of this direct sum decomposition.

Then, up to simultaneous similarity, we can assume that $\V_1=\ldots=\V_r=\mathbb{C}^{n/r}$ and that the set of non-zero elements in each $(i,j)$-block
$$\mathcal{H}_{i,j}= P_i\G P_j \setminus\{0\} \subseteq \mathrm{L}(\V_i,\V_j)$$ 
is individually equal to a fixed irreducible matrix group $\mathcal{H}\subseteq \M_{n/r}(\mathbb{C})$.  We can additionally assume that each $\V_i$ is invariant for the similarity in question.
\end{prop}
\begin{proof}  By irreducibility of $\G$ we have that each set $\H_{i,j}$ is non-empty.  Since $\G$ is a group we have that elements of $\H_{i,j}$ are invertible and therefore we must have that $\dim\V_i=\dim\V_j$ for all $i,j$.  From now on assume that $\V_1=\ldots=\V_r=\mathbb{C}^{n/r}$.

For $G\in\G$ let $G_{i,j}=P_i G P_j$ denote the $(i,j)$-block of $G$.  Let $G,H\in\G$ and $i,j,k$ be such that $G_{i,j}\not=0$ and $H_{j,k}\not=0$.  Then, due to block-monomiality we have that $G_{i,\ell}=0=H_{\ell,k}=0$ for all  $\ell\not=j$.  Hence $(GH)_{i,k}=G_{i,j}H_{j,k}$.  Hence for all $i,j,k$ we have that $\H_{i,j}\H_{j,k}\subseteq \H_{i,k}$.  Also note that for $G\in\G$ with $G_{i,j}\not=0$ we have that $(G^{-1})_{j,i}=(G_{i,j})^{-1}$ and hence for all $i,j$ we have $\H_{i,j}^{-1}\subseteq \H_{j,i}$.

We will now explain, why we can, up to a block-diagonal similarity assume that for $i=1,\ldots, r$ we have that $I_{n/s}\in\H_{1,i}$.  Fix $G^{(1)},\ldots, G^{(r)}\in\G$ such that for all $i$ we have $X_i:=(G^{(i)})_{1,i}\not=0$.  Additionally assume that $G^{(1)}=I$ and hence $X_1=I_{n/r}$.  Let $X=\diag{X_1,\ldots, X_r}$ and note that
via simultaneous similarity $G\mapsto X G X^{-1}$ we have that $X_1=X_2=\ldots=X_r=I_{n/r}$.

From now on assume that for $i=1,\ldots, r$ we have $I_{n/r}\in \H_{1,i}$.  Hence we also have that $I_{n/r}=I_{n/r}^{-1}\in \H_{i,1}$.  Let $i,j\in\{1,\ldots, r\}$. Inclusion $\H_{i,1} = \H_{i,1}I \subseteq \H_{i,1}\H_{i,1} \subseteq \H_{1,1}$ yields that $\H_{i,1}\subseteq \H_{1,1}$.  Similarly $\H_{1,i}\subseteq \H_{1,1}$.  On the other hand $\H_{1,1}I\subseteq\H_{1,1}\H_{1,i}\subseteq \H_{1,i}$ so that also $\H_{1,1}\subseteq\H_{1,i}$.  Hence $\H_{1,i}=\H_{1,1}=\H$.  Similarly $\H_{j,1}=\H$.  Now $\H_{1,i}\H_{i,j}\H_{j,1}\subseteq\H_{1,1}$ yields that $\H_{i,j}\subseteq \H_{1,1}$ and $\H_{i,1}\H_{1,1}\H_{1,j}\subseteq\H{i,j}$ yields that $\H_{1,1}\subseteq \H_{i,j}$; so that $\H_{i,j}=\H_{1,1}=\H$.
\end{proof}

An important tool in our considerations is Clifford's Theorem \cite[Theorem 1, p. 113]{S} (see also the original reference \cite{C}).  Below we state it in terms of block-monomial matrices (combined with the above proposition).

\begin{thm}[Clifford's Theorem]\label{thm-Clifford} Let $\G\subseteq\M_n(\mathbb{C})$ be an irreducible group and let $\N$ be a reducible normal subgroup such that not all irreducible representations of $\N$ on $\mathbb{C}^n$ are pairwise isomorphic (or, equivalently, there is no similarity under which $\N=I_m\otimes \N_0$ for some irreducible group $\N_0\subseteq \M_{n/m}(\mathbb{C})$).  Let $\V_1,\ldots, \V_r$ be all $\N$-invariant subspaces of $\mathbb{C}^n$ that are maximal such that for each fixed $i=1,\ldots, r$ we have that all irreducible sub-representations of $\N$ on $\V_i$ are isomorphic (as representations). 

Then $r>1$, for each $i=1,\ldots, r$, $\dim\V_i=n/r$, $\mathbb{C}=\V_1\oplus\ldots\oplus\V_r$, and
and $\G$ is block-monomial with respect to this direct sum decomposition. 

We can additionally assume, up to simultaneous similarity, that 
 for all $i,j=1,\ldots, r$, we have that the set of non-zero elements of the block $P_i\G P_j\subseteq \mathrm{L}(\V_j, \V_i)=\mathcal{M}_{n/r}(\mathbb{C})$  is equal to a fixed irreducible group $\H\subseteq\M_{n/r}(\mathbb{C})$ (here $P_i$ denotes the projection to the $i$-th summand in the direct sum decomposition $\mathbb{C}=\V_1\oplus\ldots\oplus\V_r$).
\end{thm} 
\qed

\subsection{Group actions and averaging}
Let $\G$ be a group containing an abelian normal subgroup $\D$.  Then $\G$ acts on $\D$ (on the right) by $D^G=G^{-1}DG$, for $G\in \G$ and $D\in\D$.  If $\K$ is a finite subgroup of $\G$ and $D\in\D$, then  we abbreviate $$\avg_\K(D) = \prod_{K\in\K} D^K$$ (the notation $\prod$ is unambiguous as $\D$ is commutative).  If $G\in\G$ is an element of finite order $m$, then we also write 
$$\avg_G(D)=\avg_{\langle G \rangle}(D) = D D^{G}\ldots D^{G^{m-1}}.$$  

Note that elements $G\in\G$ and $D\in \D$ commute if and only if the action of $G$ on $D$ is trivial, i.e., $D^G=D$.  Suppose now that the order of $G\in\G$ is odd and that the order of $D\in\D$ is two.  Then we have that \textsl{$D$ and $G$ commute if and only if $\avg_G(D)=D$.}  This observation will play an important role throughout the paper.

In the applications below $\G$ will be a signed permutation matrix group and $\D$ will be the subgroup of diagonal matrices in $\G$. 

\subsection{Monomial groups with no diagonal commutation}

Let $\G$ be a monomial matrix group and let $\D\subseteq\G$ be the subgroup of all diagonal matrices in $\G$.  Note that the pattern group $\Pat(\G)$ acts naturally on $\D$ as for every $G\in \G$ and $D\in \D$ we have that $D^G=D^{\Pat(G)}$.

We say that $\G$ {\em has no diagonal commutation} if every nontrivial element of the pattern group $\Pat(\G)$ acts nontrivially on every nonscalar element of $\D$.  Or, equivalently, if for every $G\in\G\backslash \D$ and every $D\in\mathcal{D}\backslash \mathbb{C}I$ we have that $GD\not=DG$. 

Fix an odd natural number $n>1$ and an indecomposable abelian permutation matrix group $\K\subseteq\Mn$.  Below we describe signed diagonal groups $\J_\K,\J_\K^+\subseteq \Mn$ that will play an important role in the paper.  We define them as follows:
$$
\J_\K=\bigl\{J\in\D_n(\pm 1) : \forall G\in\K\backslash\{I\}, \avg_G(J)=\det(J) I\bigr\},
$$
and
$$
\J_\K^+=\{J\in \J_\K : \det(J)=1\}.
$$
If $\K=\C_n$, then we abbreviate $\J_n=\J_{\C_n}$ and $\J_n^+=\J_{\C_n}^+$.
Note that 
$$\J_\K=\J_K^+\cup (-\J_K^+).$$
Observe also that $\J_\K$ is $\K$-stable and hence $\K\J_\K$ is a group (as for $J,L\in\J_n$ and $G,H\in\K$ we have that $(GJ)(HL)=(GH)(J^HL)$).  Abstractly this group is a semidirect product of $\J_\K$ and $\K$.  Below we will describe the structure of the group $\J_\K$.  We will, among other things, prove that $\J_\K$ is nonscalar if and only if $\K$ is cyclic (and therefore, up to simultaneous permutational similarity, equal to $\C_n$).

\begin{lem}
Let $G_1,\ldots, G_k$ be generators of all subgroups of $\K$ of prime order and let
$J\in\D_n(\pm 1)$.  Then $J\in\J_\K^+$ if and only if for all $i=1,\ldots, k$ we have that $\avg_{G_i}(J)=I$.
\end{lem}
\begin{proof}  If $\avg_G(J)=I$ for all $G\in\K\setminus\{I\}$, then we clearly also have that for $i=1,\ldots, k$, $\avg_{G_i}(J)=I$.

Now assume that for $i=1,\ldots, k$ we have that $\avg_{G_i}(J)=I$.  Let $G\in\K\setminus\{I\}$ be of order $m$.  Let $p$ be a prime dividing $m$.  Then $G^{m/p}$ is of order $p$ and hence there is an $i$ such that $\langle G^{m/p}\rangle = \langle G_i\rangle$.  Hence we have $\avg_{G^{m/p}}(J)=\avg_{G_i}(J)=I$ and therefore
$\avg_{G}(J)=\avg_{G^{m/p}}(J)\avg_{G^{m/p}}(J)^G\ldots \avg_{G^{m/p}}(J)^{G^{(m/p)-1}}=
I\cdot I^G\cdot\ldots \cdot I^{G^{(m/p)-1}}=I$.
\end{proof}

\begin{lem}\label{lem-powerprime} Assume that $n=p^m$ for some prime $p$, $\K=\C_n$, and $J=\diag{J_1,\ldots, J_p}$ with $J_i\in\D_{n/p}(\pm 1)$ for $i=1,\ldots, p$.  Then $J\in\J_\K^+$ if and only
if $J_1\ldots J_p=I$, or, equivalently, $J_p=J_1\ldots J_{p-1}$.
\end{lem}
\begin{proof}
Let $G=C_n^{n/p}$.  Then $G=C_p\otimes I_{n/p}$ generates the only subgroup of $\K$ of prime order and hence $J\in\J_\K^+$ if and only if $\avg_{G}(J)=I$.  Now note that
$\avg_{G}(J)=I_p\otimes (J_1\ldots J_p)$.
\end{proof}

\begin{prop}\label{prop-decomp} Assume that $n=p^m n'$, for prime $p$ and $n'$ coprime to $p$, and $\K=\C_{p^m}\otimes \K_0$ for some indecomposable abelian group $\K_0\subseteq\M_{n'}(\mathbb{C})$.  Let $J=\diag{J_1,\ldots, J_p}$ with $J_i=\diag{J_i^{(1)},\ldots, J_i^{(p^{m-1})}}$ with $J_i^{(j)}\in \D_{n'}(\pm 1)$ for $i=1,\ldots, p$ and $j=1,\ldots, p^{m-1}$.  Then $J\in \J_\K^+$ if and only if $J_1\ldots J_p=I$ and for all $i,j$ we have that $J_i^{(j)}\in\J_{\K_0}^+$.
\end{prop}
\begin{proof}  Let $G_1=C_{p}\otimes I_{p^{m-1}}\otimes I_{n'}$ and let $G_2,\ldots, G_k$ be generators of all subgroups of prime order of $\K_0$.  Then (since $p$ does not divide $n'$, the order of $\K_0$ by Lemma \ref{transitive-commutative}) we have that $G_1, I_{p^m}\otimes G_2, \ldots, I_{p^m}\otimes G_k$ are generators of all subgroups of $\K$ of prime order.  Now the conclusion follows by noting that $\avg_{G_1}(J)=I_{p}\otimes (J_1\ldots J_p)$ and that for $j=2,\ldots, k$ we have
$\avg_{I_{p^m}\otimes G_j}(J)=\diag{\tilde{J}_1,\ldots, \tilde{J}_p}$ where $\tilde{J}_i=
\diag{\avg_{G_j}(J_i^{(1)}),\ldots, \avg_{G_j}(J_i^{(p^{m-1})})}$ for $i=1,\ldots, p$.
\end{proof}

\begin{cor} We have that $|\J_n^+|=2^{\varphi(n)}$, where $\varphi$ is the Euler's totient function. 
\end{cor}
\begin{proof}  If $n$ is a power of a prime then the result follows from Lemma \ref{lem-powerprime}.  Now assume that $n$ is not a power of a prime.  Let $n=p_1^{m_1}\ldots p_k^{m_k}$ be the decomposition of $n$ into the product of pairwise distinct primes $p_1,\ldots, p_k$.  Then, up to permutational similarity, we have that $\C_{p_1^{m_1}}\otimes \K_0$, where $\K_0=\C_{p_2^{m_2}}\otimes\ldots\otimes\C_{p_k^{m_k}}$.  Now Proposition \ref{prop-decomp} gives that $\displaystyle |\J_\K^+|={|\J_{\K_0}^+|}^{p_1^{m_1}-p_1^{m_1-1}}$.  

The claim  $\displaystyle |\J_n^+|={2}^{\varphi(n)}$ now follows by induction on $k$.
\end{proof}

\begin{cor}\label{prop-Jn}
Let $n$ be odd.  Then the groups of signed diagonal matrices $\J_n^+$  and $\J_n$ 
are not scalar.
\end{cor}
\qed

\begin{lem}\label{lem-cyc} If $\K$ is not cyclic, then $\J_\K=\{\pm I\}$.
\end{lem}
\begin{proof} Assume that $\K$ is not cyclic.  Up to monomial similarity we can assume that $$\K=\C_{n_1}\otimes \ldots\otimes\C_{n_k}$$ 
with $n_1$ and $n_2$ having a common prime factor $p$.  Let $A=C_{n_1}^{n_1/p}\otimes I_{n_2}\otimes\ldots\otimes I_{n_k} = C_p\otimes I_{n_1/p}\otimes I_{n_2}\otimes\ldots\otimes I_{n_k}=C_p\otimes I_{n/p}$ and let $B=I_{n_1}\otimes C^{n_2/p}\otimes I_{n_3}\otimes\ldots\otimes I_{n_k}$.  Also let $P=I_{n_1/p}\otimes C^{n_2/p}\otimes I_{n_3}\otimes\ldots\otimes I_{n_k}$ so that $P$ is of order $p$ and $B=I_p\otimes P$.  Let $J=\diag{J_0,\ldots, J_{p-1}}\in \J_\K^+$ with $J_0,\ldots, J_{p-1}\in \D_{n/p}(\pm 1)$.  Observe that $J^A=\diag{J_1,\ldots, J_{p-1},J_0}$ and
$J^B=\diag{J_0^P,\ldots, J_{p-1}^P}$.  Now $I=\avg_{B}(J)=\diag{\avg_P(J_0),\ldots, \avg_{P}(J_{p-1})}$, so for all $i=0,\ldots, p-1$ we have $\avg_P(J_i)=I$.  Fix $i=0,\ldots, p-1$. Comparing $(k+1,k+1)$-diagonal blocks of the equality $I_n=\avg_{AB^i}(J)$ yields that $I_{n/p}=\prod_{j=0}^{p-1} J_{j+k}^{P^{ij}}$. Hence we have that for all $k$ (in the computation we use the convention that indices are taken modulo $p$, i.e., for $p\le \ell\le 2p-1$ we have $J_\ell = J_{\ell-p}$):
\begin{eqnarray*}
I_{n/p}&=& \prod_{i=0}^{p-1} I_{n/p} = \prod_{i=0}^{p-1}\prod_{j=0}^{p-1} J_{j+k}^{P^{ij}}\\
&=& \prod_{i=0}^{p-1}\left(J_k\prod_{j=1}^{p-1} J_{j+k}^{P^{ij}}\right) = J_k^p \prod_{i=0}^{p-1}\prod_{j=1}^{p-1} J_{j+k}^{P^{ij}} \\
&=& J_k^p \prod_{j=1}^{p-1}\prod_{i=0}^{p-1}  J_{j+k}^{P^{ij}} 
= J_k^p \prod_{j=1}^{p-1} I_{n/p}\\ &=& J_k^p = J_k.
\end{eqnarray*}
Hence $J=I$ and we can conclude that $\J_\K^+=\{I\}$ and $\J_\K=\{\pm I\}$.
\end{proof}

\section{Groups whose commutator subgroups consist of involutions}

The main purpose of the paper is to study irreducible semigroups in which ring commutators have real spectra.  The structure of unitary groups with this property is an important ingredient.  However, for groups, it is perhaps more natural, to study group commutators.
In this section we briefly explore the structure of compact groups in which every element of the derived subgroup has real spectrum (or, equivalently, is an involution).  The later sections of the paper will not depend on the discussion that follows. 

We start by the following well-known observation.

\begin{prop} Let $\G\subseteq\Mn$ be an irreducible group.  If its commutator subgroup $[\G,\G]$ is diagonalizable, then up to simultaneous similarity, $\G$ is monomial with commutative pattern.
\end{prop}
\begin{proof}
This is a straightforward corollary of the famous theorem of Suprunenko \cite{S} which states that every irreducible nilpotent group is monomializable.  Indeed, if the commutator subgroup $[\G,\G]$ is scalar, then $\G$ is nilpotent and hence monomializable.  If $[\G,\G]$ is not scalar, then we invoke Clifford's theorem to block-monomialize $\G$.   Now observe that the blocks are individually equal to a fixed irreducible nilpotent group (which can be monomialized by using Suprunenko's Theorem again).
\end{proof}

We abbreviate $\mathbb{C}^\times=\mathbb{C}\backslash\{0\}$. 

\begin{prop}\label{prop-commutatorsubgroup} The commutator subgroup of $\G\subseteq\Mn$ consists of involutions if and only if, up to simultaneous similarity, $\G$ is contained in $\mathbb{C}^\times \H$, where $\H$ is a signed permutation group with commutative pattern. 
\end{prop}
\begin{proof} $(\implies):$ 
With no loss we assume that $\G$ is a monomial group.  We can, and do, additionally assume that the commutator subgroup is diagonal.  This immediately yields that the pattern of $\G$ is commutative.  Due to irreducibility of $\G$ we conclude that $\Pat(\G)$ is transitive.  Suppose that $D=\diag{\alpha,\ldots,\beta,\ldots}$ belongs to $\G$.  Due to the transitivity of $\Pat(\G)$  we get that some matrix $D_1$ of the form $D_1=\diag{\alpha\beta^{-1},\ldots}$ belongs to $[\G,\G]$ and hence $\beta=\pm \alpha$ (if $\beta$ is in position $i$ and $G\in\G$ has pattern that maps $e_i$ to $e_1$, then $GDG^{-1}D^{-1}$ has the desired form).  

$(\Longleftarrow):$ Since the pattern is assumed to be commutative, we have that $[\G,\G]$ is a subset of signed diagonal matrices.
\end{proof}

Note that if $\G$ is compact, then its commutator subgroup consists of involutions if and only if the spectrum of every element of the commutator subgroup is real. 

\begin{question} Can we reach the conclusion of Proposition \ref{prop-commutatorsubgroup} above with the (at least apriori) weaker assumption that all group commutators in $\G$ are involutions?
\end{question}

Another natural question that arises from considerations above is the following.
\begin{question} When is the pattern of a monomial group $\G$, up to simultaneous (monomial) similarity, a subgroup of $\G$?
\end{question}

The following technical lemma partially addresses this question.   Recall that a group is $n$-divisible if every element is an $n$-th power.

\begin{lemma}  \label{lem-split}
Let $\G\subseteq\M_n(\mathbb{C})$ be an indecomposable  monomial group of matrices with commutative pattern and let $\D$ be the subgroup of diagonal matrices in $\G$.  If $X$ and $Y$ are subgroups of the multiplicative group of complex numbers $\Cx$ such that $\D\subseteq
Y\D_X$ (here $\D_X$ is the group of diagonal matrices in $\G$ with entries from $X$),  and $Y$ is $n$-divisible, then, up to a diagonal similarity, $\G = Y\G_X$, where $\G_X$ is a group of matrices in $\G$ with nonzero entries from $X$.  Furthermore, if the order of $X$ is coprime to $n$, then, up to a diagonal similarity, the pattern group of $\G$ is a subgroup of $\G$.
\end{lemma}
\begin{proof} Assume with no loss of generality that $Y\subseteq \G$.  Denote the pattern subgroup of $\G$ by $\PP$ and consider the exact sequence $\D\to\G\stackrel{\pi}{\to}\PP$.  Let $\PP=\langle a_1,\ldots, a_k\rangle$, where $a_i$'s are cyclic generators of $\PP$ of order
$n_i$.  Let $g\colon \PP\to\G$ be a splitting of $\pi$.  Now define a new such splitting $f\colon \PP\to\G$ by $f(a_i)=\mu_ig(a_i)$, where 
$\mu_i\in Y$ are such that $\mu_i^n g(a_i)^n\in \D_X$ and by $f(a_1^{i_1}\ldots a_k^{i_k})=f(a_1)^{i_1}\ldots f(a_k)^{i_k}$ and observe that $f$ is a homomorphism modulo $\D_X$; more precisely, we have a map $\alpha\colon \PP\times\PP\to\D_X$ such that
$f(x)f(y)=\alpha(x,y)f(xy)$ for all $x,y\in\PP$.  Now rescale the standard basis $e_x=xe_1$ (as indexed by $\PP$) by setting $\widetilde{e_x}=f(x)e_1$. The computation $f(x)\widetilde{e_y} = f(x)f(y)e_1 = \alpha(x,y) f(xy)e_1 = \alpha_{x,y}\widetilde{e_{xy}}$ shows that using this diagonal similarity we achieve the desired result (a $\G$ is generated by $f(\PP)$ and $\D=Y\D_X$).

Now assume that the order of $X$ is coprime to the order of $\G$.  Then by the Schur-Zassenhaus Theorem $\D_X$ is a complemented subgroup of $\G$.
Let $\mathcal{Q}$ denote such a complement.  Note that we have an exact sequence $(Y
\cap \mathcal{Q}) \to \mathcal{Q}\to \PP$ and that now (in a fashion almost identical to the argument above) we can choose a splitting $f\colon \PP\to (Y\cap\mathcal{Q})$ that is a group homomorphism.  The rescaling of the basis $\widetilde{e_x}=f(x) e_1$ then finishes the proof.
\end{proof}
The following example shows that the $n$-divisibility of $Y$ is crucial:
\begin{example}
Let $\G$ be the subgroup of $3\times 3$ matrices generated by $\xi C_3$ and all diagonal matrices of the form $\diag{\pm 1, \pm 1, \pm 1}$, where $\xi$ is a primitive ninth root of unity.  Then no diagonal similarity can possibly force any element of the form
$\diag{\pm 1, \pm 1, \pm 1}C_3$ to belong to $\G$.  In this case we also have that the order of $X=\{-1, 1\}$ is coprime to $n=3$, and under no diagonal similarity we have that the pattern of $\G$ is a subgroup of $\G$.
\end{example}
The following example shows that if the order of $X$ is not coprime to $n$, then even with the existence of an $n$-divisible $Y$, we may not be able to find a diagonal similarity under which the pattern of $\G$ becomes a subgroup of $\G$.
\begin{example}
Let $\K$ be the set of all $2\times 2$ matrices of the form $\diag{\pm 1, \pm 1}C_2$ and $G$ be the group of $6\times 6$ matrices generated by all nonzero scalars, the matrix $I_2\otimes C_3$, and all block diagonal matrices of the form $M_{A,B,C}=\diag{A, B, C}$ where $A,B,C\in\K$ are such that $\det(ABC)=1$.  Note that the square of no scalar multiple of any $M_{A,B,C}$ is scalar and hence no such matrix can be diagonally similar to its pattern (the square of the pattern of $M_{A,B,C}$ is $I$).
\end{example}

\begin{prop}  Suppose that $\mathcal{G}=\mathbb{C}^\times \mathcal{G}\subseteq\Mn$ is an irreducible group whose commutator subgroup consists of involutions.
If $n$ is odd, then, up to similarity, $\G=\mathbb{C}^\times\PP\ltimes\J$, where $\PP$ is an indecomposable commutative permutation group and $\J=[\G,\G]$ is a $\PP$-stable nonscalar subgroup of signed diagonal matrices. 
\end{prop} 
\begin{proof}
By Proposition \ref{prop-commutatorsubgroup} and the fact that $\G=\mathbb{C}^\times\G$ we get that $\G=\mathbb{C}^\times\H$, where $\H$ is an indecomposable signed permutation group with commutative pattern.  Now use Lemma \ref{lem-split} with $X=\mathbb{C}^\times$ and $Y=\{-1,1\}$.
\end{proof}

\begin{remark} If $\G$ is compact, then we can replace $\mathbb{C}^\times$ by the unit circle $\{z\in\mathbb{C}:|z|=1\}$ to get the analogous conclusion. 
\end{remark}

\section{Structure of compact groups of matrices in which all ring commutators have real spectra}

The main result of this section is the following theorem.  

\begin{thm}\label{thm-maingrp}  Let $\G\subseteq\Mn$ be an irreducible compact group.  Then the following are equivalent.
\begin{enumerate}
\item All ring commutators $AB-BA$, $A,B\in\G$, have real spectra.
\item Number $n$ is odd and, up to simultaneous similarity $\G=\C_n\D$ for some nonscalar $\C_n$-stable subgroup $\D$ of $\J_n$. 
\end{enumerate}
\end{thm}

We will need several technical results, in addition to earlier discussion, before we can start with the proof.  But first, let us state the following corollary which will be needed in the last section.

\begin{cor}\label{cor-main}  Let $\G$ be an irreducible compact group in which all ring commutators have real spectra.  Then, up to simultaneous similarity, $\G$ is s signed permutation group with commutative pattern.  In particular $\G$ is realizable.
\end{cor}
\qed

The key ingredient of the proof of Theorem \ref{thm-maingrp} is the following result from
\cite{MR}.
\begin{thm}[cf. {\cite[Theorem 3.05 ]{MR}}]\label{noncentral}
If $\mathcal{G}\subseteq \Mn$ is a nonabelian compact group of matrices
such that every ring commutator $ST-TS$, $S, T\in\mathcal{G}$ has real
spectrum, then $\mathcal{G}$ contains a noncentral involution.
\end{thm}
\qed

\begin{lemma}\label{lem-nil} Let $n$ be odd and let $\G=\C_n\J_n$. If $G\in\G$ is not diagonal and $X,Y\in\G$ are diagonal elements of equal determinants, then $(X-Y)G$ is nilpotent.
\end{lemma}
\begin{proof}  With no loss of generality we assume that $G\in \C_n$ (if necessary, replace $\G$ by $\Pat(\G)$, $X$ by $XG\Pat(G)^{-1}$, and $Y$ by $YG\Pat(G)^{-1}$) and that $\det(X)=1=\det(Y)$ (if necessary, replace $X$ by $-X$ and $Y$ by $-Y$).  Let $m$ be the order of $G$.  Note that $m$ is odd as it must divide $n$.
We now compute
\begin{eqnarray*}
((X-Y)G)^m &=& G^{-m}((X-Y)G)^m = G^{-m}(X-Y)G ((X-Y)G)^{m-1} \\
&=& G^{-m}(X-Y)G^m G^{-(m-1)}((X-Y)G)^{m-1} \\
&=& (X-Y)^{G^m} G^{-(m-1)}((X-Y)G)^{m-1}=\ldots \\
&=& (X-Y)^{G^m} (X-Y)^{G^{m-1}}\ldots (X-Y)^G \\
&=& \prod_{H\in\langle G\rangle} (X-Y)^H \\
&=& \sum_{\mathcal{A}\subseteq \langle G\rangle} (-1)^{m-|\A|}\prod_{A\in \A} X^A \prod_{B\in \langle G\rangle\backslash \A} Y^B.
\end{eqnarray*}
For every $\A\subseteq\langle G\rangle$ we have that $\prod_{A\in \A} X^A \prod_{B\in \langle G\rangle\backslash \A} X^B=
\prod_{C\in \langle G\rangle} X^C=I$ and hence $\prod_{A\in \A} X^A=\prod_{B\not\in\A} X^B$ (as $X^{-1}=X$).
Since $\langle G\rangle$ is of odd order, we therefore have that in the above sum the terms corresponding to $\A$ and $\langle G\rangle\backslash \A$ cancel
and the sum is thus $0$.
\end{proof}

\begin{lemma}\label{involutions}
Let $\mathcal{G}$ be a compact group of matrices and let $\J$ be the set
of all involutions in $\G$.  If all ring commutators of elements of $\G$
have real spectra then $\J$ is a commutative normal subgroup of $\G$.
\end{lemma}
\begin{proof} Note that it is sufficient to prove that $\J$ is a
commutative set.  Now suppose, if possible, that there exist a pair $J,K$
of noncommuting involutions in $\G$. With no loss of generality assume
that $\G$ is a group of unitary matrices, that $J=\begin{pmatrix} I & 0\\
0 & -I\end{pmatrix}$ (diagonal blocks are of nonzero, possibly different,
sizes), and that $K=\begin{pmatrix} A & B \\ C & D\end{pmatrix}$.  Since
$K^*=K$ (by assumption $K^2=I=KK^*$) we get that $C^*=B$.  A routine
computation shows that $([J, K]_r)^2=(JK-KJ)^2=-4\begin{pmatrix} BB^* &
0\\ 0 & B^*B\end{pmatrix}$.  Since $[J,K]_r=JK-KJ$ has real spectrum we
conclude that $([J, K]_r)^2$ has a nonnegative real spectrum and thus we
must have $C=B=0$.  But then $JK=KJ$, contradicting our initial
assumption.
\end{proof}

We now proceed with the proof of Theorem \ref{thm-maingrp}.
\begin{proof}(of \ref{thm-maingrp}).

(2)$\implies$ (1): Assume that $\G=\C_n\D$, where $\D$ is a non-scalar subgroup of $\J_n$.  Let $A,B\in \G$.  Note that $\Pat(AB)=\Pat(BA)$.  If
$AB$ is diagonal, then so is $BA$ and hence $\sigma(AB-BA)\subseteq \{-2,0,2\}$.  If $AB$ is not diagonal, then
$AB-BA=(ABA^{-1}B^{-1}-I)BA$.  Now apply Lemma \ref{lem-nil} with $X=ABA^{-1}B^{-1}, Y=I$ and $G=BA$ to conclude that $\sigma(AB-BA)=\{0\}$.

(1)$\implies$ (2): Let $\G$ be an irreducible compact group of $n\times n$ matrices such that
all ring commutators $ST-TS$ of elements $S,T\in\G$ have real spectra.
Let $\J$ be the set of all involutions in $\G$.  Recall that by Lemma
\ref{involutions} $\J$ is a commutative normal subgroup of $\G$.  With no
loss we assume from now on that $\J$ is a subset of diagonal matrices.

Let $\mathbb{C}=\V_1\oplus\ldots\oplus\V_r$ be the weight space decomposition with respect to the action of $\J$ on $\mathbb{C}^n$, i.e., $\V_i$'s are maximal $\J$-invariant subspaces of $\mathbb{C}^n$ such that the restrictions $\J|_{\V_i}$ are scalar.  Let $P_1,\ldots, P_r$ be the projections to the corresponding summands in this direct sum decomposition.

By Clifford's Theorem, we have that the spaces $\V_1, \ldots,
\V_r$ are all of dimension equal to $s=n/r$, that $\G$ acts transitively on the set $\{\V_1,\ldots, \V_r\}$, and that $\G$ is block monomial with respect to the decomposition $\CC^n=\V_1\oplus\ldots\oplus \V_r$.   Abbreviate the
irreducible group $\{G\in\G: G(\V_1)\subseteq\V_1\}|_{\V_1}=P_1\G P_1\setminus\{0\}$ of nonzero elements in the $(1,1)$-block of this block decomposition by $\H$.  Form here on we also assume, with no loss, that all blocks $P_j\G P_i\setminus\{0\}$ of $\G$ are individually equal to $\H$ (see Proposition \ref{prop-block}).

Also note that by Theorem \ref{noncentral} we have that $r\not=1$ and hence $s<n$ and observe that for every pair of distinct integers $p,q$, $1\le p,q\le r$ there is an element $J\in\J$ such that its $(p,p)$-block is the negative of its $(q,q)$-block.
We now proceed in small steps.

\noindent\textrm{STEP ONE}: \textsl{$\G$ is monomializable.}

This is proven by induction.  The statement is clear for $n=1$ and also
for $s=1$.  Now assume that $1<n,s$ and that for $m<n$ all $m\times m$
irreducible compact matrix groups with ring commutators having real
spectra are monomializable.  Hence $\H$ is monomializable and hence so is
$\G$.

From now on assume that $\G$ is monomial.

\noindent\textrm{STEP TWO:} \textsl{All diagonal elements of $\G$ are
involutions (and thus belong to $\J$).}

If $s>1$, then we can (using induction) assume that the statement holds
for $\H$ and then it must automatically also hold for $\G$.
Suppose now that $s=1$ and let $D=\diag{d_1,\ldots, d_n}$.  Pick
$i\in\{2,\ldots, n\}$, let $J\in\J$ be such that
$J_{11}=-J_{ii}$, and let $G\in\G$ be such that $Ge_i\in \mathbb{C} e_i$.
Note that the $(1,1)$ entries of the diagonal matrices
$G D G^{-1}-D=[GD,G^{-1}]_r$ and $G(JD)G^{-1}-JD=[GJD,G^{-1}]_r$ are $d_i-d_1$ and $\pm(d_i+d_1)$.
Since these entries must be real we deduce that $d_i$ is real
and thus equal to $\pm 1$.

\noindent\textrm{STEP THREE:} \textsl{$\G$ is finite.}

Let $G\in\G$.  Note that $G^{n!}$ is a diagonal matrix as $\G$ is
monomial.  By the argument above all diagonal matrices in $\G$ are involutions and
hence $\G^{2n!}=I$.  So $\G$ is an irreducible matrix group of finite exponent and is
thus finite.  (This follows, e.g., from \cite{RR1}, since the trace functional, when restricted to $\G$ has a finite number of values.)

\noindent\textrm{STEP FOUR:} \textsl{$\G$ contains no elements of order
$4$.}

We do a proof by contradiction.  Suppose $G\in\G$ is such that $G^4=I$ and
$G^2\not=I$.  We use induction to assume that if $s>1$, then $\H$ has no
elements of order $4$.  This implies that $G$ cannot be block diagonal (if
$s=1$ this fact follows from Step Two above).
Hence we can assume, using a similarity by a block permutation if
necessary, that the compression $G_0$ of $G$ to $\V_1 \oplus \V_2$  has the
form
$$
G_0=\begin{pmatrix} 0 & X \\ Y & 0\end{pmatrix},
$$
with $G_0^2\not=I$.
Let $J\in\J$ be such that its compression $J_0$ to $\V_1\oplus \V_2$ is
given by
$$
J_0=\pm\begin{pmatrix} I & 0\\ 0 & -I\end{pmatrix}.
$$
Now note that $(J_0 G_0)J_0^{-1} - J_0^{-1}(J_0 G_0) = J_0 G_0 J_0^{-1} -
G_0 = -2G_0$ does not have real spectrum.  A contradiction.

\noindent\textrm{STEP FIVE:} \textsl{$\J$ is complemented in $\G$}

Since $\G$ contains no elements of order four we conclude that $\G/\J$
contains no elements of order two and hence $m:=\left|\G/\J\right|$ is
odd.  By the Schur-Zassenhaus Theorem $\J$ is
complemented, that is there exists a subgroup
$\K\le \G$ of order $m$ such that $\G=\K\ltimes\J$.  That is, $\K\cap
\J=\{I\}$ and $\K\J = \G$.

\noindent\textrm{STEP SIX:} \textsl{$s=1$.}

If $s>1$ then $\H$ contains a noncentral involution $J_0$.  Let $G\in\G$
be an element whose $(1,1)$-block is equal to $J_0$.  Then
$G^m$ (where, as in Step Five, $m=\left|\G/\J\right| $) is an involution
whose $(1,1)$ block is $J_0$.  This is impossible since, by construction,
the blocks of elements of $\J$ can only be $\pm I$.

\noindent\textrm{STEP SEVEN:} \textsl{$\K$ is commutative.}

If $\K$ were not commutative, then it would contain a noncentral
involution.  This is impossible since the order of $\K$ is odd.

\noindent\textrm{STEP EIGHT:} \textsl{$\K$ is, up to monomial similarity,
a permutation group (equal to tensor product of
cyclic groups).}  The claim follows from Lemma \ref{transitive-commutative}.

\noindent\textrm{STEP NINE:} \textsl{$\G$ has no diagonal commutation, that is, if $G\in \G\backslash \J$ and
$J\in\J\backslash \{\pm I\}$, then $JG\not=GJ$.}  We do a proof by contradiction.  Suppose, if possible, that $G\in \G$ is non-diagonal, and $J\in \J$ is nonscalar such that $GJ=JG$.  We assume with no loss that $G\in\K\setminus\{I\}$. (Any non-diagonal $G\in \G\setminus\J$ is of the form $G=J_1G_1$ for some $J_1\in\J$ and $G_1\in\K\setminus\{I\}$.  Since $J_1$ and $J$ commute we have that $J$ and $G_1$ must also commute.  So we can replace $G$ by $G_1$ if necessary.)

For $H\in\K$ we define $B_H=J^H-J$ and 
$A_H=B_H G = H^{-1}(JG)H -JG = [H^{-1}, JGH]_r$.  Note that the set $\mathcal{B}=\{B_H:H\in\K\}$ is a commuting set (it is a subset of diagonal matrices) and since all members of $\mathcal{B}$ commute with $G$ ($J$ commutes with $G$ by assumption; every $H\in\K$ commutes with $G$ as $\K$ is abelian) we have that the set $\mathcal{A}=\{A_H:H\in\K\}$ is a commuting set.  Therefore $\mathcal{A}$ is simultaneously triangularizable.  Every member of $\mathcal{A}$ is a ring commutator of elements from $\G$ and hence has real spectrum.  Since $\mathcal{A}$ is simultaneously triangularizable we therefore have that all $\mathbb{R}$-linear combinations of its elements also have real spectra.  Since the action of $\mathcal{K}$  on $\{\mathbb{C}e_1,\ldots, \mathbb{C}e_n\}$ is transitive, there must exists an $\mathbb{R}$-linear combination $B$ of members of $\mathcal{B}$ whose diagonal entries are all nonzero.  But then $A=BG$ must have real spectrum as it is an $\mathbb{R}$-linear combination of elements of $\mathcal{A}$.  But this is impossible.  Indeed, up to a permutational similarity (corresponding to decomposition of the permutation associated to $G$ into disjoint cycles) we have that $BG=B_1C_{n_1}\oplus\ldots \oplus B_kC_{n_k}$ with $B_1, \ldots , B_k$ invertible diagonal matrices.  Since $G$ is of odd order, we have that all $n_i$'s are odd.  Let $i$ be such that $n_i>1$ (it exists, as $G\not=I$).  Now observe that the spectrum of $B_iC_{n_i}$ (which is contained in the spectrum of $BG$) is equal to $\{\lambda\in\mathbb{C}: \lambda^{n_i}=\det B_i\}$ and hence not real.

\noindent\textrm{CONCLUSION:} \textsl{Up to simultaneous similarity, we have that $\G=\C_n\J$, where $\J$ is a $\C_n$-stable nonscalar subgroup of $\J_n$.}  We have already established that $\G=\K\J$.  The fact that $\K=\C_n$ follows from Lemma \ref{lem-cyc}.  The fact that $\J\subseteq \J_n$ follows from the fact that $\G$ has no diagonal commutation (established in the Step Nine). 
\end{proof}

\section{Semigroups of matrices in which all ring commutators have real spectra are realizable}
\begin{lemma}\label{rank-one}
Irreducible rank-one semigroups whose commutators have real spectra are
realizable.
\end{lemma}
\begin{proof}
Without loss assume that the semigroup is real-homogenized and
closed.  We proceed by contradiction.  Assume, if possible that, the
semigroup is not realizable.  Then there is a member whose spectrum is not
real (as rank one-semigroups with real spectra are realizable \cite{BMR}).
This member must be of the form $\lambda E$ for  some idempotent $E$.
Since the semigroup is homogeneous and closed we conclude that $E$ belongs
to it as well \cite{RR}.  Since for any $S$ in the semigroup we have that
the spectra of $[E,S]$ and $[\lambda E, S]=\lambda [E,S]$ are real, we
conclude that for all $S$ in the semigroup we have that the commutator
$[E,S]$ is nilpotent.  From now on assume, using a simultaneous similarity
if necessary, that
$
E=\begin{pmatrix}
1 & 0 \\
0 & 0
\end{pmatrix}.$
Write two general members $A$ and $B$ as
$$ A=\begin{pmatrix}
x & Y \\
Z & T\end{pmatrix}
\mbox{ and } B=
\begin{pmatrix}
x' &  Y'\\
Z' &  T'\end{pmatrix}.$$
Since the ring commutator $[E, A]$ is nilpotent, we conclude that
$$-[E, A]^2=\begin{pmatrix}
YZ & 0 \\
0  & ZY \end{pmatrix}$$
has zero spectrum
and thus $YZ= 0$.  We shall now show that $YZ'$ is real for any two
members. Now $$BEA =
\begin{pmatrix}
x'x  &  x'Y \\
xZ'  &  Z'Y
\end{pmatrix},$$
so  $xx'YZ' = 0$ by the above.  If $xx'$ is not zero, then $YZ' = 0$, so
we can
assume $xx' = 0$.   Use the reality of the spectrum of $[EA, BE]$, to get
that the square of $[BE,EA]=\begin{pmatrix}
-YZ' &  x'Y \\
   xZ' &  Z'Y\end{pmatrix} $
has positive spectrum. But the trace of this square is
$2(YZ')^2$. Thus $YZ'$ is real as claimed.
Now by irreducibility there is a member $B$ with $Z'$ nonzero. Let
$L=\begin{pmatrix}
0  & 0 \\
Z'  & 0
\end{pmatrix},$
so that the rank-one linear functional $f$ on $n\times n$ matrices defined
by
$f(M) = tr ML$
always has real values on our semigroup. This implies realizability by
Proposition  2.4 of \cite{BMR1}.
\end{proof}

\begin{thm}If $\S\subseteq\M_n(\mathbb{C})$ is an irreducible semigroup in
which all
ring commutators $ST-TS$, $S,T\in\S$ have real spectra, then $\S$ is
realizable.
\end{thm}
\begin{proof}
With no loss of generality assume that $\S=\overline{\RR^+\S}$.  Let $E$
be the minimal rank idempotent in $\S$.  If the rank of $E$ is
one, then the rank-one ideal in $\S$ is realizable by Lemma \ref{rank-one}
and hence so is $\S$ (because the rank-one linear functional on $\M_n$ defined by $M\mapsto\tr(ME)$, when restricted to $\S$ has real values and thus we can apply Theorem 2.5 of
\cite{RY}; this also follows from Corollary 3.5 of \cite{BMR}).

If the rank of $E$ is larger then one, then $E\S E|_{E(\mathbb{C}^n)} =
\RR^+\G$, where $\G$ is a compact group.  By Corollary \ref{cor-main}
$\G$ is realizable and hence so is $\S$ (by \cite{RY} again).
\end{proof}


\begin{thebibliography}{00}
\bibitem{B} J.~Bernik, \textsl{The eigenvalue field is a splitting field}, Arch. Math 88 (2007), 481--490.

\bibitem{BMR} J.~Bernik, M.~Mastnak, H.~Radjavi, \textsl{Realizing
semigroups and real algebras of compact operators}, J. Math. Anal. Appl.
348 (2008), no. 2, 692--707.

\bibitem{BMR1}  J.~Bernik, M.~Mastnak, H.~Radjavi, \textsl{Positivity and 
matrix semigroups}, Lin. Alg. Appl. 434 (2011), no. 3, 801--812.

\bibitem{C} A.~H.~Clifford, \textsl{Representations induced in an
invariant subgroup}, Ann. Math. (2) 38, 533--550.

\bibitem{MR} M.~Mastnak, H.~Radjavi, \textsl{Structure of finite, minimal
nonabelian groups and triangularization}, Lin. Alg. Appl. 430 (2009),  Issue 7, 1838--1848.

\bibitem{RR} H.~Radjavi, P.~Rosenthal, \textsl{Simultaneous
triangularization}, Springer, New York, 2000.

\bibitem{RR1} H.~Radjavi, P.~Rosenthal, \textsl{Limitations on the size of semigroups of matrices}, Semigroup Forum 76 (2008), 25--31.

\bibitem{RY} H.~Radjavi, B.~Yahaghi, \textsl{On irreducible algebras spanned by triangularizable matrices}, Lin. Alg. Appl. 436 (2012), 2001--2017.

\bibitem{S} D.~A.~Suprunenko, \textsl{Matrix Groups}, AMS Translations of Math. Monographs, Vol. 45, 1976, p.219.
\end{thebibliography}
\end{document}